\documentclass[12pt]{amsart}
\usepackage{amssymb,amsxtra,amsmath,amsfonts}
\usepackage{verbatim}
\usepackage{graphicx,color,graphics} 
\usepackage{xypic}
\usepackage{braket}
\usepackage[all]{xy}
\usepackage{mathtools}
\usepackage[colorlinks = true,
   linkcolor = blue,
 urlcolor  = green,
citecolor = green,
anchorcolor = blue]{hyperref}
\theoremstyle{plain}

\newtheorem*{theorem*}{Theorem}
\newtheorem{theorem}{Theorem}

\newtheorem{lemma}[theorem]{Lemma}
\newtheorem{proposition}[theorem]{Proposition}
\newtheorem*{proposition*}{Proposition}
\newtheorem{corollary}[theorem]{Corollary}
\newtheorem*{corollary*}{Corollary}

\newtheorem{thm}{Theorem}[section] 
\theoremstyle{plain} 
\newcommand{\thistheoremname}{}
\newtheorem{genericthm}[thm]{\thistheoremname}

\theoremstyle{definition}
\newtheorem{definition}[theorem]{Definition}

\theoremstyle{remark}
\newtheorem{remark}[theorem]{Remark}

\numberwithin{equation}{section}
\numberwithin{theorem}{section}

\renewcommand{\comment}[1]{}

\makeatletter
\@namedef{subjclassname@2020}{\textup{2020} Mathematics Subject Classification}
\makeatother
\def\CC{\mathbb{C}}

\def\Y{\mathcal{Y}}

\def\V{\mathcal{V}}

\def\ZZ{\mathbb{Z}}

\DeclareMathOperator\End{End}

\DeclareMathOperator\Hom{Hom}

\def\vac{{\boldsymbol{1}}}  

\def\lieg{{\mathfrak{g}}}

\begin{document}

\title[Associated varieties of modules of quiasi-lisse vertex algebras]{ Associated Varieties of Ordinary Modules over Quasi-Lisse Vertex Algebras}

\author{Juan Villarreal}
\address{Department of Mathematics, 
CU Boulder, 
CO 80309,
United States}
\email{juanjos3villarreal@gmail.com}


\begin{abstract}
We prove that if $V$ is a conical simple self-dual quasi-lisse vertex algebra and $M$ is an ordinary module then $\dim X_M=\dim X_V$. Hence, if moreover $X_V$ is irreducible then $X_M=X_V$. 

In particular, this applies to quasi-lisse simple affine vertex algebras $L_{k}(\lieg)$. For admissible $k$ it reproves a result in \cite{A2}, and it further extends it to non-admissible levels.

\end{abstract}

\maketitle


\section{Introduction}

The associated variety of a vertex algebra \cite{Z,A} is a fundamental geometric invariant that captures key features of the representation theory \cite{A2,ADF+,AdV,AnEH,AFK,AJM,AK,BFM, JS, KG, Liu, LS,LS2,Mi, PXZ}; for a survey, see \cite{AM}. The concept has also played a role in physics \cite{BR,BMR,PXY}. 

In this paper we study the relations between associated varieties of modules over a vertex algebra and fusion (intertwining operators) among those modules. For a $V$-module $M$ we write $X_M$ for its associated variety (see §\ref{sec 2.2}). Our first main statement is:

\begin{proposition*} Let $V$ be a conical vertex algebra and let $M_1, M_2, M_3$ be ordinary modules. If there is a surjective intertwiner of type $\binom{M_3}{M_1\quad M_2}$, then 
\[\dim X_{M_3}\leq \text{min}\{ \dim X_{M_1},  \dim X_{M_2}\}+1  \]
\end{proposition*}

See Proposition~\ref{pro3.3}, the proof of this proposition uses the theory of Hilbert generating functions in commutative algebra.

Then, we obtain a dimension statement for quasi-lisse algebras:

\begin{theorem*} Let $V$ be a conical self-dual simple quasi-lisse vertex algebra and $M$ a simple ordinary module
\begin{itemize}
\item[i] If V is self-dual then $\dim X_{M}=\dim X_V$.
\item[ii] If V is self-dual and $X_V$ is irreducible then $X_V=X_M$.
\end{itemize}
\end{theorem*}

Details are given in Theorem~\ref{thm3.5}. In particular, the theorem applies to quasi-lisse simple affine vertex algebras and affine $W$-algebras. We remark that if $V$ is quasi-lisse then it has finite simple ordinary modules, see \cite{AK, ES}. Furthermore, it has been conjectured in \cite{AM2} that the associated variety $X_V$ of a simple conical quasi-lisse vertex algebra is irreducible. Hence, assuming this conjecture, the theorem  implies that for such vertex algebras, all ordinary modules have the same associated variety. See also \cite{L, DF}.

For simple affine vertex algebras $L_{k}(\mathfrak{g})$ at $k$ admissible level \cite{KW}, using a different method, it was previously established in \cite{A2} that $X_{V}$ is a quasi-lisse irreducible variety (explicitly determined) and $X_{M}=X_{V}$ for ordinary simple modules. Beyond admissible levels we have for example that for $L_{-2}(G_2), L_{-2}(B_3)$ the associate varieties are quasi-lisse irreducible, with non-trivial (i.e. $M\neq L_{k}(\mathfrak{g})$) simple ordinary modules \cite{ADF+}[Theorem B, Theorem C].


\section{Preliminars} \label{pre}

Let $V$ be a complex vector space. We denote by $V [\![z]\!]$  the space of formal power series with coefficients in $V$, and by  $V(\!(z)\!):=V [\![z]\!][z^{-1}]$ the formal Laurent series. More generally, $V[\![z]\!]z^{-\CC}$ denotes the space of sums $\sum_{n}f_{n}z^{n}$, where $f_n\in V$ and $n\in \{d_{i} + \ZZ_{\geq0}\}$ for finitely many $d_i \in  \CC$. Note that $V [\![z]\!]z^{-\CC}$ has the usual action of the derivative $\partial_{z}$.

Now, we review some standard definitions, we follow \cite{AM, K, FB}.

\begin{definition} \label{d3.11}
A vertex algebra is a vector space $V$ equipped with a vector $\vac\in V$  and a linear map 
\[Y\colon  V\to \Hom(V,V(\!(z)\!))\,, \qquad v\mapsto Y(v,z)=\sum_{n\in \ZZ}v_{(n)}z^{-n-1}\,, \]
subject to the following axioms:

\medskip
(\emph{i})\; $Y(\vac,z)=I_{V}$, \; $Y(v,z)\vac\in V[\![z]\!]$, \; $Y(v,z)\vac\big|_{z=0} = v$. 


(\emph{ii})\;  $\forall$ $u,v\in V$ and $n\in \ZZ$ 
\begin{equation}\label{b1}
\begin{split}
\iota_{z_1, z_2}z_{12}^{n}Y(v&,z_1)Y(u, z_2)-\iota_{z_2, z_1}z_{12}^{n}Y(u, z_2)Y(v,z_1)b\\
&=\sum_{i\geq 0}Y( v_{(n+i)}u_j, z_2) \partial^{(i)}_{z_2}\delta(z_1, z_2)\, .
\end{split}
\end{equation}
\end{definition}

On a vertex algebra we define $T\in \End(V)$ by $Tv=v_{(-2)}\vac$, then by definition we have $[T,Y(v,z)] = \partial_z Y(v,z)$. A \emph{Hamiltonian} of a vertex algebra is a locally finite semisimple operator $H$ on $V$ satisfying 
\begin{equation}\label{hami}[H, v_{(n)}]=-(n+1)v_{(n)}+(Hv)_{(n)}.\end{equation}
 A vertex algebra with a Hamiltonian is called graded, we denote $V_{\Delta}:=\{v\in V | Hv=\Delta v\}$ for $\Delta\in \CC$. A vertex algebra is positive graded if $V=\bigoplus_{\Delta\in\frac{1}{r_0}\ZZ_{\geq 0}} V_{\Delta}$ where $r_0$ is a positive integer. A positive graded vertex algebra is called \emph{conical} if $V=\bigoplus_{\Delta\in\frac{1}{r_0}\ZZ_{\geq 0}} V_{\Delta}$  and $V_0=\CC$.

\begin{definition} \label{d3.11}
A $V$-module is a vector space $M$ equipped with a linear map 
\[Y^{M}\colon  V\to \Hom(M,M(\!(z)\!))\,, \qquad v\mapsto Y^{M}(v,z)=\sum_{n\in \ZZ}v^{M}_{(n)}z^{-n-1}\,, \]
subject to the following axioms:

\medskip
(\emph{i})\; $Y^{M}(\vac,z)=I_{M}$. 


(\emph{ii})\;  $\forall$ $u,v\in V$ and $n\in \ZZ$
\begin{equation}\label{b1}
\begin{split}
\iota_{z_1, z_2}z_{12}^{n}Y^M(v&,z_1)Y^M(u, z_2)-\iota_{z_2, z_1}z_{12}^{n}Y^M(u, z_2)Y^M(v,z_1)\\
&=\sum_{i\geq 0}Y^M( v_{(n+i)}u, z_2) \partial^{(i)}_{z_2}\delta(z_1, z_2)\, .
\end{split}
\end{equation}
\end{definition}

By definition it follows that $Y^M(Tv, z)= \partial_z Y^M(v,z)$. 
A $V$-module $M$ is graded if $V$ is graded and there is a semisimple operator $H$ on $M$  such that \eqref{hami} is satisfied on $M$. Let $M_{\Delta}:=\{m\in M | Hm=\Delta m\}$. 

\begin{definition}We say that $M$ is a positively graded if there is a $h\in \CC$ such that $M=\bigoplus_{\Delta\in \{h+\frac{1}{r_0}\ZZ_{\geq 0},0\}}M_{\Delta}$. We use the notation $M^{{top}}=M_h$, the top component of $M$. A positive energy representation $M$ is called an ordinary representation if each $M_{\Delta}$
is finite-dimensional.  \end{definition}

Additionally, we have the definition of intertwiner from \cite{FHL}

\begin{definition}\label{defint} Let $M_1 , M_2, M_3$ three $V$-modules. 
 An intertwining of type $\binom{M_3}{M_1\quad M_2}$ is a linear map 
\[\Y:M_1 \rightarrow \Hom(M_2,M_3[\![z]\!]z^{-\CC})\,, \qquad a\mapsto \Y(a,z)=\sum_{n\in \CC}a_{(n)} \/\, z^{-n-1}\, , \]
subject to the following axioms:

\medskip

(\emph{i})\; 
$\Y(Ta,z)  = \partial_z \Y(a,z)$.

(\emph{ii})\;  $\forall$ $v\in V, a\in M_1$ and $n\in \ZZ$
\begin{equation}\label{b}
\begin{split}
\iota_{z_1, z_2}z_{12}^{n}Y(v&,z_1)\Y(a, z_2)-\iota_{z_2, z_1}z_{12}^{n}\Y(a, z_2)Y(v,z_1)\\
&=\sum_{i\geq 0}\Y( v_{(n+i)}a, z_2)\partial^{(i)}_{z_2}\delta(z_1, z_2)\, .
\end{split}
\end{equation}
\end{definition}
We say that the intertwiner is graded if $M_1, M_2, M_3$ are graded and 
\begin{equation}\label{graint} [H, a_{(k)}]=-(k+1)a_{(k)}+(Ha)_{(k)}\end{equation}

 Note that a $V$-module $M$ defines an intertwiner of type $\binom{M}{V \quad  M}$. 

\begin{remark} We note that the identity \eqref{b} is equivalent to ($n,m\in \ZZ$, $k\in \CC$) the following identity 
\begin{equation}\label{borcherds2}
\begin{split}
\sum_{j\in \ZZ_+} &(-1)^{j}\binom{n}{j}\left(v_{(m+n-j)}a_{(k+j)} -(-1)^{n}a_{(n+k-j)} v_{(m+j)}\right) \\
& =\sum_{j\in \ZZ_+} \binom{m}{j}( v_{(n+j)}a )_{(m+k-j)} \,.
\end{split}
\end{equation}

\end{remark}

  \subsection{Poisson algebras and Poisson modules} \label{sec 2.2}
  
   \begin{definition} Let $R$ be a Poisson algebra. A Poisson module is an $R$-module $M$ in the usual associative sense equipped with a bilinear map 
 \[\{\cdot,\cdot \}\colon  R\otimes M\rightarrow M\,, \]
subject to the following axioms: For all $r,s\in \V$, $m\in M$  

\medskip
(\emph{i})\;  $\{r ,\{ s, m\}=\{s, \{r, m\}\}+\{ \{r,s\},m\}$.
\medskip

(\emph{ii})\;  $\{r, sm\}=\{r,s\}m+s\{r,m\}$.
\medskip

(\emph{iii})\;  $\{rs, m\}=r\{s,m\}+s\{r,m\}$.
 \end{definition}

For a vertex algebra $V$ and a module $M$, we define the subspaces $F^{1}V=V_{(-2)}V$, $F^{1}M=V_{(-2)}M$. 

  \begin{theorem}\label{thm2.9} $R_V:=V/F^1V$ has a natural poisson algebra structure and 
   $\bar{ M}:=M/F^{1}M$ 
 forms a  Poisson $R_{V}$-module.
 \end{theorem}
 
The Poisson algebra $R_V$ is the {Zhu's $C_2$-algebra} \cite{Z}.  $R_V$ is defined as follows: For $u, v\in V$ 
 \begin{align*}
 \bar{u}\cdot \bar{v}:=\overline{u_{(-1)}v}\, , \qquad \{\bar{u}, \bar{v}\}=\overline{u_{(0)}v}\, ,\, 
 \end{align*}
 where $\bar{u}:=\sigma_0(u)$.  And the Poisson module  $\bar{ M}$ is defined as follows: For $u\in V, m\in M$
  \begin{align*}
 \bar{u}\cdot \bar{m}:=\overline{u_{(-1)}m} ,\qquad \{\bar{u}, \bar{m}\}=\overline{u_{(0)}m} ,\, 
 \end{align*}
  where $\bar{m}:=\sigma_0(m)$.

 
A vertex algebra $V$ is called \emph{finitely strongly generated} if $R_V$ is finitely generated as a ring. A $V$-module $M$ is called \emph{finitely strongly generated} if $\overline{M}$ is finitely generated as a $R_V$-module. 


In \cite{A}, the associated variety of a vertex algebra $V$ and a module $M$ were defined as follows
 \begin{align}
 &X_{V}:=\text{specm}(R_V)\, , \\
&X_{M}:=\text{supp}_{R_V}(\bar{M})\subset X_{V}\, .
\end{align}
Where $\text{suppm}$ means the maximal ideals in $\text{supp}$.

\section{Associated varieties and Intertwiners}

In this section, we assume all intertwiners are graded \eqref{graint}. The identity \eqref{borcherds2} gives us the following two identities
\begin{align}
&[v_{(m)},a_{(k)}]b  =\sum_{j\in \ZZ_+} \binom{m}{j}( v_{(j)}a )_{(m+k-j)} b\, , \label{1}\\
&( v_{(-1)}a )_{(k)}b=\sum_{j\in \ZZ_{\geq 0}} \left(v_{(-1-j)}a_{(k+j)}b +a_{(-1+k-j)} v_{(j)}b\right)\label{2}  \,.
\end{align}

\begin{lemma}\label{lem3.1} Let $V$ be a conical vertex algebra and let $M_1, M_2, M_3$ be positive energy modules. For a surjective intertwiner of type $\binom{M_3}{M_1\quad M_2}$, we have as vector spaces 
\begin{itemize}
\item[i)] $M_3/F^{1}M_3\subset \left(\bigoplus_{r\geq 0}{M_1^{\text{top}}}_{(k-r)}(M_2/F^{1}M_2)\right)/F^{1}M_3$,
\item[ii)] $M_3/F^{1}M_3\subset \left(\bigoplus_{r\geq 0}({M_1/F^{1}M_1})_{(k-r)}M_2^{\text{top}}\right)/F^{1}M_3$,

\end{itemize}
where $k=h_1+h_2-h_3-1$.
\end{lemma}

\begin{proof} It follows from \eqref{graint} that for $k=h_1+h_2-h_3-1$ we have a surjective linear map
\[{M^{\text{top}}_1}\otimes_{\CC}{M^{\text{top}}_2}\rightarrow {{M^{\text{top}}_1}}_{(k)}{M^{\text{top}}_2}\cong {M^{\text{top}}_3}\, .\]

We prove i) in two steps. 

Step 1: $M_3/F^1M_3\subset \left(\bigoplus_{r\geq 0} { M_1^{\text{top}}}_{(k-r)}M_2\right)/F^{1}M$. For an homogeneous vector $v\in V$, $a\in  M_1^{\text{top}}, b\in M_{2}^{\text{top}}$, we have from \eqref{1} that
\begin{align*} v_{(-1)}a_{(k)}b=a_{(k)}v_{(-1)}b+\sum\binom{-1}{j} (v_{(j)}a)_{(k-j-1)}b\, .
\end{align*}
The first term on the right-hand side belongs to ${ M_1^{\text{top}}}_{(k)}M_2$. For the second term, we have a finite sum $v_{(j)}a=\sum {v^{i}_j}_{(-1)}a_i$ for $v^{i}_j\in V, a_{i}\in  M_1^{\text{top}}$, then from \eqref{2}
\begin{align*}
& (\sum {v^{i}_j}_{(-1)}a_i)_{(k-j-1)}b=\\\
&\qquad \sum {v^{i}_j}_{(-1)} (a_i)_{(k-j-1)}b+\sum_{r\in {\ZZ}_{\geq 0}}(a_i)_{(k-j-2-r)}{v^{i}_j}_{(r)}b, \quad \text{mod }F^1M_3\, .
\end{align*}
The second term on the right hand side is in ${\bigoplus_{r\geq 0} M_1^{\text{top}}}_{(k-2-r)}M_2$, and the first terms satisfy $\Delta_{v^i_j}<\Delta_v$. Hence, because $V$ is conical if we repeat the previous two-steps a finite number of times on the first term we obtain $ (v_{(j)}a)_{(k-j-1)}b\in \bigoplus_{r\geq 0}{ M_1^{\text{top}}}_{(k-1-r)}M_2$ mod $F^1M_3$.

Step 2: If $m\in F^1M_2$ homogeneous of weight $\Delta_m$ then from \eqref{1} and \eqref{2}
\[a_{(k)}m=a_{(k)}v_{(-2)}m'=-\sum_{j\in \ZZ_+} \binom{-2}{j}( v_{(j)}a )_{(k-2-j)} m'\]
\[=\sum {v^{i}_j}_{(-1)} (a_i)_{(k-j-2)}b+\sum_{l\in {\ZZ}_{\geq 0}}(a_i)_{(k-j-3-l)}{v^{i}_j}_{(l)}b, \quad \text{mod }F^1M_3\, .\]
The second term on the right hand side belongs to $\bigoplus_{r\geq 0}{ M_1^{\text{top}}}_{(k-3-r)}M_2$ and $\Delta_{{v^{i}_j}_{(l)}b}<\Delta_{m}$. The first term on the right hand side belongs to $\bigoplus_{r\geq 0}{ M_1^{\text{top}}}_{(k-2-r)}M_2$ by step 1, hence we have a finite sum ${v^{i}_{j}}_{(-1)} (a_i)_{(k-j-2)}b=\sum_r \sum_s{a^r_s}_{(k-2-r)}m^r_s$ for $a^r_s\in M^{\text{top}}_1$, $m^r_s\in M_2$ and $\Delta_{m^r_s}<\Delta_{m}$. Because $M_2$ is a positive energy module if we repeat the previous steps a finite number of times we obtain i).

The proof of ii) is analogous.
\end{proof}
If $V$ is graded then $HF^1V\subset F^{1}V$.  Hence if $V$ is {conical} then  	
\begin{equation}\label{conical}R_{V}=\bigoplus_{\Delta\in\frac{1}{r_0}\ZZ_{\geq 0}}{(R_V)}_{\Delta}\, , \qquad {(R_V)}_{0}=\CC\, .\end{equation}
Analogously, if $M$ is positively graded $HF^1M\subset F^1M$, then 
 \[\bar{M}=\bigoplus_{\Delta\in\frac{1}{r_0}\ZZ_{\geq 0}}{(\bar{M})}_{\Delta+h}\]
We have the following Corollary.
\begin{corollary}\label{cor3.2} Let $V$ be a conical vertex algebra and let $M_1, M_2, M_3$ be ordinary modules. For a surjective intertwiner of type $\binom{M_3}{M_1\quad M_2}$, we have for $m\in \frac{1}{r_0}\ZZ_{\geq 0}$ 
\begin{itemize}
\item[i)] $\dim \bar{M}^3_{h_3+m}\leq (\dim  M_1^{\text{top}})(\dim \bar{M}^2_{h_2+m-\lfloor m\rfloor}+\cdots +\dim \bar{M}^2_{h_2+m})\, ,  $\\
\item[ii)] $\dim \bar{M}^3_{h_3+m}\leq (\dim \bar{M}^1_{h_1+m-\lfloor m\rfloor}+\cdots +\dim \bar{M}^1_{h_1+m})(\dim  M_2^{\text{top}})\, .  $
\end{itemize}
\end{corollary}
\begin{proof}From Lemma \ref{lem3.1} i), we have from the grading that ${\bar{M}^{3}_{h_3+m}}\subset {M^{\text{top}}_1}_{(k)}{\bar{M}^{2}_{h_2+m}}+{M^{\text{top}}_1}_{(k-1)}{\bar{M}^{2}_{h_2+m-1}}+\cdots +{M^{\text{top}}_1}_{(k-\lfloor m \rfloor)}{\bar{M}^{2}_{h_2+m-\lfloor m\rfloor}}$, where $\lfloor \cdot \rfloor$ denote the floor function as usual. Hence, we obtain the Corollary i).  Analogously for Corollary ii).
\end{proof}





Now, we change the grading of $R_{V}$ in \eqref{conical} as follows ${(\tilde{R})_V}_{r_0\Delta} ={(R_V)}_{\Delta}$, then $R_{V}=\bigoplus_{n\geq 0} {(\tilde{R})_V}_{n}$. Analogously, ${(\tilde{{M}})}_{r_0\Delta}={(\bar{M})}_{\Delta+h}$ and $\bar{M}=\bigoplus_{n\geq 0}{(\tilde{{M}})}_{n}$.  Then from Corollary \ref{cor3.2}, we have in particular
\begin{itemize}
\item[j)] $\dim \tilde{M}^3_{n}\leq (\dim  M_1^{\text{top}})(\dim \tilde{M}^2_{0}+\cdots +\dim \tilde{M}^2_{n})\, ,  $\\
\item[jj)] $\dim \tilde{M}^3_{n}\leq (\dim \tilde{M}^1_{0}+\cdots +\dim \tilde{M}^1_{n})(\dim  M_2^{\text{top}})\, .  $
\end{itemize}
Note that the dimensions of $\text{Supp}\bar{M}_1,\text{Supp}\bar{M}_2$ and $\text{Supp}\bar{M}_3 $ are independent of the grading. \\

 Let $R=\bigoplus_{n\geq 0} R_n$ be a a graded commutative ring such that $R_0=\CC$ and $R$ is a finitely generated $\CC$-algebra. For a finitely generated graded module $M=\bigoplus_{n\geq 0} M_n$, the \emph{Hilbert generating function} is $H_{M}(t):=\sum_{n\geq 0}\left(\dim_{\CC} M_n\right) t^n$. Let $v_i\in R_{r_i}$ be finite generators of $R$, from Hilbert-Serre theorem 
 \[H_{M}(t)=\frac{Q(t)}{\Pi_i(1-t^{r_i})}, \]
{where} $Q(t)\in \ZZ[t]$, see \cite{AtMD, AlKl}. Moreover, the order of pole of $H_{M}(t)$ at $t=1$ is given by $\dim \text{Supp} M$, \cite{S}. 
\begin{proposition}\label{pro3.3} Let $V$ be a conical vertex algebra and let $M_1, M_2, M_3$ be ordinary modules. If there is surjective intertwiner of type $\binom{M_3}{M_1\quad M_2}$, then 
\[\dim X_{M_3}\leq \text{min}\{ \dim X_{M_1},  \dim X_{M_2}\}+1  \]

\end{proposition}
\begin{proof} From j) above, we have the following relation (the inequality is coefficient-wise)
\begin{align*}
H_{\bar{M}_3}(t)&=\sum_{i\geq 0} \dim \tilde{M}^3_i t^{i} \leq (\dim  M_1^{\text{top}})\sum_{i\geq 0} (\sum_{j=0}^i\dim \tilde{M}^2_j) t^{i}\\
&=(\dim  M_1^{\text{top}})(\sum_{j\geq 0}\dim \tilde{M}^2_j t^j)\sum_{i\geq 0}  t^{i}=(\dim  M_1^{\text{top}})\frac{1}{1-t}H_{\bar{M}_2}(t)
\end{align*}
Therefore the order of the pole at $t=1$ of $H_{\bar{M}_3}(t)$ is less or equal than the order of the pole at $t=1$ of $\frac{1}{1-t}H_{\bar{M}_2}(t)$. Hence, $\dim X_{M_3}\leq  \dim X_{M_2}+1$. Analogously, using jj) we have $\dim X_{M_3}\leq  \dim X_{M_1}+1$.
\end{proof}

\begin{proposition}\label{pro3.4}
If $V$ is simple conical self-dual and $M$ is ordinary then $\dim X_M =\dim X_V $ or $\dim X_M =\dim X_V -1$.  
\end{proposition}
\begin{proof} We have that $\binom{M}{V \quad M}\cong \binom{V}{M \quad M'}$ where $M'$ is the contragredient, in particular $M'$ is simple ordinary if $M$ is simple ordinary. Moreover, a non-trivial intertwiner of type $\binom{V}{M \quad M'}$ must be surjective, because $V$ is simple. Then $\dim X_V \leq \dim X_{M}+1$.  Finally, $X_{M}\subset X_V$ hence we have two possibilities $\dim X_M =\dim X_V $ or $\dim X_M =\dim X_V -1$.  
\end{proof}


Now, $V$ is quasi-lisse if it is conformal and $X_V$ has finitely simpletic leaves. If $M$ is ordinary then $\text{Supp}\bar{M}=V(Ann_{R_{V}}\bar{M})$ and $Ann_{R_{V}}\bar{M}$ is a Poisson ideal, hence only even dimensional supports are allowed. We obtain the following theorem

\begin{theorem}\label{thm3.5} Let $V$ be a conical simple quasi-lisse vertex algebra and $M$ a simple ordinary module
\begin{itemize}
\item[i] If V is self-dual then $\dim X_{M}=\dim X_V$.
\item[ii] If V is self-dual and $X_V$ is irreducible then $X_V=X_M$.
\end{itemize}
\end{theorem}



\begin{thebibliography}{DHVW}





\bibitem[AbBD]{AbBD}
 Abe, T., Buhl, G., Dong, C.: 
 \textit{Rationality, regularity, and C2-cofiniteness.} 
 Trans. Am. Math. Soc. 356(8), 3391–3402 (2004)
 
 \bibitem[AdV]{AdV}
  Adamovich, D., Vukorepa, I.: 
 \textit{A new quasi-lisse affine vertex algebra of type $D_4$}
\url{ https://arxiv.org/pdf/2504.13783}

\bibitem[AnEH]{AnEH}
Andrews, G., van Ekeren, J., Heluani, R.
\textit{The singular support of the Ising model}
Int. Math. Res. Not.,  10, 8800–8831 (2023).

 \bibitem[AlKl]{AlKl}
 Altman, A.,  Kleiman, S.
 \textit{A Term of Commutative Algebra}
 Worldwide Center of Mathematics, 2012.
 
\bibitem[A]{A}
Arakawa, T.:
\textit{A remark on the $C_2$-cofiniteness condition on vertex algebras}.
Math. Z. \textbf{270}, 559–575 (2012).

\bibitem[A2]{A2}
Arakawa, T.:
\textit{Associated varieties of modules over Kac–Moody algebras and $C_2$-Cofiniteness of W-algebras} 
Int. Math. Res. Not.,  22, 11605–11666 (2015).

\bibitem[ADF+]{ADF+}
Arakawa, T., Dai, X.,  Fasquel, J., Bohan Li and Moreau, A. 
\textit{On some simple orbifold affine VOAs
at non-admissible level arising from rank one 4D SCFTs}
 Commun. Math. Phys. 406, 30 (2025).
 
\bibitem[AFK]{AFK}
Arakawa, T., Futorny, V.,  Kriszka, L.
\textit{Generalized Grothendieck’s simultaneous resolution and associated varieties of simple affine vertex algebras}.
\url{https://arxiv.org/abs/2404.02365}

\bibitem[AJM]{AJM}
Arakawa, T., Jiang, C., and Moreau, A. 
\textit{Simplicity of vacuum modules and associated varieties}. 
J. Ec. polytech. Math. \textbf{8}, 169–191 (2021)

\bibitem[AK]{AK}
Arakawa, T., Kawasetsu, K.
\textit{Quasi-lisse vertex algebras and modular linear differential equations}.
Birkhäuser/Springer, Cham, 41–57 (2018).





\bibitem[AM]{AM}
Arakawa, T., Moreau, A.:
\textit{Arc spaces and vertex algebras} preprint

\bibitem[AM2]{AM2}
Arakawa, T., Moreau, A.:
\textit{Sheets and associated varieties of affine vertex algebras}
{Adv. Math.} \textbf{ 320},  7, 2017.

\bibitem[AM3]{AM3}
Arakawa, T., Moreau, A.:
\textit{Joseph ideals and lisse minimal W-algebras.}
 J Inst Math Jussieu. 2018;17(2)


\bibitem[AtMD]{AtMD}
 Atiyah, M., MacDonald, I.
 \textit{Introduction To Commutative Algebra}
 Avalon Publishing, 1994.

 

 



\bibitem[BFM]{BFM}
Beilinson, A., Feigin, B., Mazur, B.
\textit{Notes on conformal field theory}.
unpublished 1999.

 \bibitem[BR]{BR}
 Beem, C., Rastelli, L. 
\textit{Vertex operator algebras, Higgs branches, and modular differential equations}.
 J. High Energ. Phys. 2018, 114 (2018).
 
 \bibitem[BMR]{BMR}
Bonetti, F.,  Meneghelli, C.,  Rastelli, L. 
\textit{VOAs labelled by complex reflection groups and 4d SCFTs}
J. High Energ. Phys. 2019, 155 (2019).

\bibitem[BV]{BV}
Bakalov, B., Villarreal, J.:
\textit{Logarithmic vertex algebras}.


\bibitem[BV2]{BV2}
Bakalov, B., Villarreal, J.:
\textit{Logarithmic vertex algebras and non-local poisson vertex algebras}.


\bibitem[DF]{DF}
 Dhillon, G. , Faergeman, J.
\textit{Singular support for G-categories}, 
https://arxiv.org/pdf/2410.18360.
 
 
 \bibitem[ES]{ES}
 Etingof, P.,  and Schedler, T.:
 \textit{Poisson traces and D-modules on Poisson varieties} 
 Geom. Funct. Anal., 20(4):958–987, 2010. With an appendix
by Ivan Losev.





 
















\bibitem[FB]{FB}
Frenkel, E., Ben-Zvi, D.:
\textit{Vertex algebras and algebraic curves}.
Math. Surveys and Monographs, 88,
Amer. Math. Soc., Providence, RI, 2001; 2nd ed., 2004



 

\bibitem[FM]{FM}
Frenkel, I.B., Malikov, F.G. 
\textit{Annihilating ideals and tilting functors}.
 Funct Anal Its Appl 33, 106–115 (1999)
 
 \bibitem[FM2]{FM2}
 Frenkel, I.B., Malikov, F.G. 
\textit{Kazhdan-Lusztig tensoring and Harish-Chandra categories}

\bibitem[FHL]{FHL}
Frenkel, I.B., Huang,  Y., Lepowsky, J.,  Meurman, A.:
\textit{On Axiomatic Approaches to Vertex Operator Algebras and Modules}.
Mem Am Math Soc., 104, 1993











\bibitem[J]{J}
Jantzen, J. 
\textit{Einhüllende Algebren halbeinfacher Lie-Algebren}
Ergebnisse der Math., Vol. 3, Springer,
New York, Tokio etc., 1983.

\bibitem[JS]{JS}
Jiang, C., Song, J.:
\textit{Associated Varieties of Simple Affine VOAs $L_{k}(\mathfrak{sl}_3)$ and W-algebras $W_k(\mathfrak{sl}_3,f)$}
Commun.Math.Phys. 406, 104 (2025) 



\bibitem[K]{K}
Kac, V.G.:
\textit{Vertex algebras for beginners}. 
University Lecture Series, 10, 
Amer. Math. Soc., Providence, RI, 1996; 2nd ed., 1998



\bibitem[KL]{KL}
Kazhdan, D. Lusztig, G.
\textit{Tensor Structures Arising from Affine Lie Algebras. I}
JAMS, Vol. 6, No. 4 1993, pp. 905-947

\bibitem[KRR]{KRR}
Kac, V.G., Raina, A.K., Rozhkovskaya, N.:
\textit{Bombay lectures on highest weight representations of infinite dimensional Lie algebras}. 
2nd ed., Advanced Ser. in Math. Phys., 29. 
World Sci. Pub. Co. Pte. Ltd., Hackensack, NJ, 2013

\bibitem[KG]{KG}
Kac, V.G., Gorelli, M.:
\textit{On simplicity of vacuum modules}
{Adv. Math.} \textbf{ 211},  2, 2007.

\bibitem[KW]{KW}
Kac, V.,  Wakimoto, M.:
\textit{ On rationality of W-algebras}
 Transform. Groups,
13, 671–713, 2008.
















\bibitem[Li]{Li}
Li, H.:
\textit{Abelianizing vertex algebras}. Comm. Math. Phys., Vol. \textbf{259}, No. 2, pp. 391–411, 2005.

\bibitem[Liu]{Liu}
Liu, J.:
\textit{One-point restricted conformal block and the fusion rules}
arXiv:2411.06313 

\bibitem[L]{L}
Losev, I., 
\textit{Harish-Chandra bimodules over quantized symplectic singularities}
Trans. Groups 26, 565–600 (2021). 

\bibitem[LS]{LS}
Linshaw, S., Song, B.:
\textit{Coset of free field algebras via arc spaces}. Int. Math. Res. Not., No. 1, 47–114, 2024.

\bibitem[LS2]{LS2}
Linshaw, S., Song, B.:
\textit{Standard monomials and invariant theory for arc spaces I: general linear group}
 Commun. Contemp. Math. vol. 26, no. 4 (2024)

\bibitem[Mi]{Mi}
Miyamoto, M.
\textit{Intertwining operator and $C_{2}$ -cofiniteness of modules}

\bibitem[PXY]{PXY}
Pan, Y., Xie, D.,  Yan, W.
\textit{Mirror symmetry for circle compactified 4d $A_1$ class-$S$ theories}
arXiv.2410.15695

\bibitem[PXZ]{PXZ}
Pan, Y.,  Yan, W., Zhao, Q.
\textit{Cyclotomic levels and associated varieties}
arxiv.2507.09254

	

\bibitem[S]{S}
Smoke, W.
\textit{Dimension and multiplicity for graded algebras}
Journal of Algebra, Volume 21, Issue 2, May 1972.



\bibitem[V]{V}
Villarreal, J.:
\textit{Lambda brackets and intertwiners}. 
arXiv:2310.18872.





\bibitem[Z]{Z}
Zhu, Y.:
\textit{Modular invariance of characters of vertex operator algebras}. J. Am. Math. Soc. Vol. \textbf{9}, No. 1, 1996.



\end{thebibliography}
\subsection*{Acknowledgments}
Part of this work was presented at two conferences: ``Vertex Algebras and Related Topics" at Ningbo University and ``Geometry, Integrability, and Symmetry" at the University of Denver. I am very grateful to the organizers of both events. I am also very grateful to Tomoyuki Arakawa for insightful discussions. This research was conducted at the University of Colorado Boulder, and I am very grateful for the excellent working conditions provided.
\bibliographystyle{amsalpha}

\end{document}